\RequirePackage{fix-cm}
\documentclass[smallextended]{svjour3} 
\smartqed 
\usepackage{mathptmx} 
\usepackage{graphicx}
\usepackage[T1]{fontenc}
\usepackage{amsfonts}
\usepackage{amsmath, amssymb}

\newtheorem{prop}{Proposition}
\newtheorem{thrm}[prop]{Theorem}
\newtheorem{defn}[prop]{Definition}
\newtheorem{cor}[prop]{Corollary}
\newtheorem{lemma1}[prop]{Lemma}

\newcommand{\rd}{\mathrm d}

\newcommand{\ts}{TS^p_\alpha}

\newcommand{\cond}{\stackrel{d}{\rightarrow}}
\newcommand{\conv}{\stackrel{v}{\rightarrow}}
\newcommand{\conw}{\stackrel{w}{\rightarrow}}

\newcommand{\lgg}{\mathrm{log}}

\journalname{Journal of Theoretical Probability}

\begin{document}

\title{Inversions of L\'evy Measures and the Relation Between Long and Short Time Behavior of L\'evy Processes}
\titlerunning{Inversions of L\'evy Measures}   
\author{Michael Grabchak}

\date{Received: date / Accepted: date}

\institute{M. Grabchak\at
              The University of North Carolina at Charlotte\\
9201 University City Blvd, Charlotte, NC 28223-0001\\
              \email{mgrabcha@uncc.edu}
}

\maketitle

\begin{abstract}
The inversion of a L\'evy measure was first introduced (under a different name) in \cite{Sato:2007}.  We generalize the definition and give some properties.  We then use inversions to derive a relationship between weak convergence of a L\'evy process to an infinite variance stable distribution when time approaches zero and weak convergence of a different L\'evy process as time approaches infinity.  This allows us to get self contained conditions for a L\'evy process to converge to an infinite variance stable distribution as time approaches zero.  We formulate our results both for general L\'evy processes and for the important class of tempered stable L\'evy processes.  For this latter class, we give detailed results in terms of their Rosi\'nski measures.

\keywords{Inversions of L\'evy Measures \and Tempered Stable Distributions \and Long and Short Time Behavior \and L\'evy Processes}
\subclass{60G51 \and 60F05  \and  60E07 }
\end{abstract}

\section{Introduction}

Let $\{X_t:t\ge0\}$ be a $d$-dimensional L\'evy process. The long (short) time behavior of the process is the weak limit of $X_t$, under appropriate shifting and scaling, as $t$ approaches infinity (zero). An alternate, but equivalent, definition, in terms of weak convergence of certain time rescaled processes, is also sometimes used (see e.g.\ \cite{Rosinski:2007}). Since L\'evy processes are a generalization of sums of iid random variables, it is not difficult to see that the long time behavior of the process corresponds to the stable distribution to whose domain of attraction $X_1$ belongs. Necessary and sufficient conditions for this are given in \cite{Rvaceva:1962} and \cite{Meerschaert:Scheffler:2001}. On the other hand, short time behavior has no simple analogue with the summation of iid random variables. Never-the-less, in certain situations, one can construct another L\'evy process such that the asymptotic behavior of the new process as time approaches infinity determines the asymptotic behavior of the original process as time approaches zero.

To construct such a process, observe that the long time behavior of a L\'evy process is governed by the tails of its L\'evy measure. In a similar way, we will see that its short time behavior is governed by the behavior of its L\'evy measure near zero. Intuitively, this means that the new process should have a L\'evy measure, which inverts the original L\'evy measure turning its behavior near zero to behavior near infinity and its behavior near infinity to behavior near zero.

A transformation of this type was introduced in \cite{Sato:2007} in the context of studying integrals with respect to L\'evy processes.  There, for any infinitely divisible distribution $\mu$ with no Gaussian part, the dual distribution of $\mu$ was defined.  This was renamed the inversion of $\mu$ in \cite{Sato:2012} and \cite{Sato:Ueda:2012}. We will refer to the L\'evy measure of the inversion of $\mu$ as the $0$-inversion of the L\'evy measure of $\mu$. We will then generalize this to what we term the $\beta$-inversion of the L\'evy measure of $\mu$, where $\beta\in[0,2]$.

Inversions of infinitely divisible distributions were used  in \cite{Sato:Ueda:2012} to derive asymptotic results for L\'evy processes.  Specifically, it was shown that if $\{X_t:t\ge0\}$ and $\{X_t':t\ge0\}$ are L\'evy processes such that the distribution of $X_1'$ is the inversion of the distribution of $X_1$ then short time convergence of $\{X_t:t\ge0\}$ to a point mass corresponds to long time convergence of $\{X_t':t\ge0\}$ to a point mass. In other words, \cite{Sato:Ueda:2012} uses inversion to show a relationship between the long and short time weak laws of large numbers. In this paper, we will use it to show a relationship between the long and short time central limit theorem for convergence to an infinite variance stable distribution.

There has been particular interest in the study of long and short time behavior in the class of tempered stable L\'evy processes. Tempered stable distributions were introduced in \cite{Rosinski:2007} as a class of models that (under certain conditions) look like infinite variance stable distributions in some central region, but they have lighter tails. This makes them particularly attractive for a variety of applications, including mathematical finance, physics, computer science, and biostatistics (see the references in \cite{Grabchak:2012a}).  An explanation of why such models appear in applications is given in \cite{Grabchak:Samorodnitsky:2010}. Sufficient conditions for the long time behavior of tempered stable L\'evy processes to be Gaussian and for the short time behavior to be the stable distribution that is being tempered are given in \cite{Rosinski:2007}, and (for certain extensions of these models) in \cite{Rosinski:Sinclair:2010} and \cite{Bianchi:Rachev:Kim:Fabozzi:2011}. 

We will be concerned with the more general class of $p$-tempered $\alpha$-stable distributions introduced in \cite{Grabchak:2012a}. The L\'evy measure of a $p$-tempered $\alpha$-stable distribution can be parametrized in terms of its so called Rosi\'nski measure. It is often easier to work with the Rosi\'nski measure than to work with the L\'evy measure directly.  For this reason all of our results for $p$-tempered $\alpha$-stable distributions are given in terms of their Rosi\'nski measures.  In fact, it is the particular structure of Rosi\'nski measures that motivates the extension of inversions to $\beta$-inversions.

In the next section we introduce our notation and give some background. In Section \ref{sec: Duals}, we define $\beta$-inversions and give some of their properties. Then, in Sections \ref{sec: sequences} and \ref{sec: long and short} we present, in parallel, convergence results for distributions in $ID_0$ and those in $\ts$. Specifically, in Section \ref{sec: sequences}, we relate convergence of a sequence of distributions in $ID_0$ ($\ts$) with the convergence of a sequence of distributions whose L\'evy  (Rosi\'nski) measures are $\beta$-inversions of the L\'evy (Rosi\'nski) measures of the original sequence. Finally, in Section \ref{sec: long and short}, we use $\beta$-inversions to derive necessary and sufficient conditions for L\'evy processes and $p$-tempered $\alpha$-stable L\'evy processes to converge to infinite variance stable distributions as time approaches zero.

\section{Preliminaries}

Let $\mathbb R^d$ be $d$-dimensional Euclidean space, let $|\cdot|$ be the usual norm on $\mathbb R^d$, let $\mathbb R^d_0=\mathbb R^d\setminus\{0\}$, and let $\mathbb S^{d-1}=\{x\in\mathbb R^d: |x|=1\}$. Let $\mathfrak B(\mathbb R^d)$ denote the Borel sets on $\mathbb R^d$ and let $\mathfrak B(\mathbb S^{d-1})$ denote the Borel sets on $\mathbb S^{d-1}$.  We will write $X\sim\mu$ to denote that $X$ is a random variable on $\mathbb R^d$ with distribution $\mu$. If $f$ and $g$ are real-valued functions, $c\in\mathbb R$, and $a\in\{0,\infty\}$, we write $f(t)\sim cg(t)$ as $t\to a$ to denote $f(t)/g(t)\to c$ as $t\to a$.  If $\rho\in\mathbb R$, $a\in\{0,\infty\}$, and $f$ is regularly varying at $a$ with index $\rho$, that is if for any $x>0$ $\lim_{t\to a}f(tx)/f(t)=x^\rho$, we write $f\in RV^a_\rho$ (for details about regular variation see \cite{Bingham:Goldie:Teugels:1987}). For $\beta\in[0,2]$, let $\mathfrak M^\beta$ be the class of Borel measures on $\mathbb R^d$ such that $M\in\mathfrak M^\beta$ if and only if
\begin{eqnarray}
M(\{0\})=0 \mbox{ and } \int_{\mathbb R^d}\left(|x|^2\wedge |x|^\beta\right)M(\rd x)<\infty.
\end{eqnarray}
Note that if $0<\beta_1<\beta_2<2$ then $\mathfrak M^{2}\subsetneq\mathfrak M^{\beta_2}\subsetneq\mathfrak M^{\beta_1}\subsetneq \mathfrak M^0$. The class $\mathfrak M^0$ is the class of all L\'evy measures. If $M_0,M_1,M_2,\dots\in\mathfrak M^0$, we write $M_n\conv M_0$ on $\mathbb R^d_0$ to mean that for any bounded, continuous Borel function $f:\mathbb R^d\mapsto\mathbb R$, which vanishes on a neighborhood of zero and on a neighborhood of infinity, we have $\int_{\mathbb R^d}f(x)M_n(\rd x)\to\int_{\mathbb R^d}f(x)M_0(\rd x)$ as $n\to\infty$.

Recall that the characteristic function of an infinitely divisible distribution $\mu$ can be written as $\hat\mu(z) = \exp\{C_{\mu}(z)\}$ where
\begin{eqnarray}\label{eq: inf div char func}
C_{\mu}(z) = -\frac{1}{2}\langle z,Az\rangle + i\langle b,z\rangle + \int_{\mathbb R^d}\left(e^{i\langle z,x\rangle}-1-i\frac{\langle z,x\rangle}{1+|x|^2}\right)M(\rd x),
\end{eqnarray}
$A$ is a symmetric nonnegative-definite $d\times d$ matrix, $b\in\mathbb R^d$, and $M\in\mathfrak M^0$. The measure $\mu$ is uniquely identified by the L\'evy triplet $(A,M,b)$ and we write $\mu=ID(A,M,b)$. If $A=0$ we write $\mu=ID_0(M,b)$.  Let $ID_0$ be the class of all infinitely divisible distributions with L\'evy triplets of the form $(0,M,b)$. It is well known that this class is not closed under weak convergence. However, we can characterize weak convergence within $ID_0$ by the following specialization of Theorem 8.7 in \cite{Sato:1999}.

\begin{lemma1} \label{lemma: conv ID}
If $\mu_n=ID_0(M_n,b_n)$ for $n=0,1,2\dots$ then $\mu_n\conw\mu_0$ if and only if $M_n \conv M_0$ on $\mathbb R^d_0$, $b_n\rightarrow b_0$, 
\begin{eqnarray*}\label{eq: gaussian comp}
\lim_{\epsilon\downarrow0}\limsup_{n\rightarrow\infty} \int_{|x|<\epsilon}|x|^2 M_n(\rd x)=0, \mbox{ and }
\lim_{N\rightarrow\infty}\limsup_{n\rightarrow\infty} \int_{|x|>N} M_n(\rd x)=0.
\end{eqnarray*}
\end{lemma1}

An important subclass of $ID_0$ is the class of $p$-tempered $\alpha$-stable distributions introduced in \cite{Grabchak:2012a}. This is an extension of the tempered stable distributions of \cite{Rosinski:2007} and \cite{Bianchi:Rachev:Kim:Fabozzi:2011}. If we allow these distributions to have a Gaussian part then we would have the class $J_{\alpha,p}$ defined in \cite{Maejima:Nakahara:2009}. For the remainder of this paper, fix $p>0$, $\alpha\in(-\infty,2)\setminus\{0\}$, and define $\gamma=\alpha\vee0$. A distribution $\mu=ID_0(M,b)$ is called $p$-tempered $\alpha$-stable, and is said to belong to class $\ts$, if
\begin{eqnarray}
M(A) = \int_{\mathbb S^{d-1}}\int_0^\infty 1_A(ru)q(r^p,u)r^{-1-\alpha}\rd r \sigma(\rd u), & A\in\mathfrak B(\mathbb R^d),
\end{eqnarray}
where $\sigma$ is a finite Borel measure on $\mathbb S^{d-1}$ and $q:(0,\infty)\times\mathbb S^{d-1}\mapsto(0,\infty)$ is a Borel function such that, for all $u\in\mathbb S^{d-1}$,    $q(\cdot,u)$ is completely monotone and $\lim_{r\to\infty}q(r,u)=0$. In \cite{Grabchak:2012a} it is shown that we can write
\begin{eqnarray}\label{eq:levy m}
M(A) = \int_{\mathbb{R}^d}\int_0^\infty 1_A(tx)t^{-1-\alpha}e^{-t^p}\rd t R(\rd x), & A\in\mathfrak{B}(\mathbb R^d)
\end{eqnarray}
for some measure $R\in \mathfrak M^\gamma$. Moreover, for fixed $p>0$ and $\alpha<2$, $R$ and $M$ uniquely determine each other.  We call $R$ the Rosi\'nski measure of $\mu$, and  we write $\mu=\ts(R,b)$. 

In \cite{Grabchak:2012a} and \cite{Maejima:Nakahara:2009}, the class $\ts$ is also defined when $\alpha=0$. However, the conditions on the Rosi\'nski measure are somewhat more complicated and we do not consider this case here. Never-the-less, some (but not all) of our results can be extended to this case, see Remark \ref{remark: case alpha=0} below.

A probabilistic interpretation of $R$ is given in \cite{Maejima:Nakahara:2009} (see also \cite{Jurek:2007} and Section 4 in \cite{Sato:2012}).  There it is shown that if $\mu= \ts(R,b)$ then under mild conditions (these always hold for $\alpha<1$) $\mu$ is the distribution of
$$
\int_0^{c_{\alpha,p}} G^*_{\alpha,p}(t) \mathrm dX_t
$$ 
where $G^*_{\alpha,p}(t)$ is the inverse function of $G_{\alpha,p}(u)=\int_u^\infty x^{-1-\alpha}e^{-x^p}\rd x$, $c_{\alpha,p}=G_{\alpha,p}(0)$, and $\{X_t:t\ge0\}$ is a L\'evy process such that the distribution of $X_1$ has L\'evy measure $R$. 

As with the class $ID_0$, the class $\ts$ is not closed under weak convergence. The smallest class that contains $\ts$ and is closed under weak convergence is characterized in \cite{Grabchak:2012b}.  The following is a specialization of a result from that paper.

\begin{lemma1}\label{lemma: conv TS}
If $\mu_n = TS^p_\alpha(R_n,b_n)$ for $n=0,1,2\dots$ then $\mu_n\conw\mu_0$ if and only if $R_n\conv R_0$ on $\mathbb R^d_0$, $b_n\rightarrow b_0$, 
\begin{eqnarray*}\label{eq: gaussian comp TS}
\lim_{\epsilon\downarrow0}\limsup_{n\rightarrow\infty} \int_{|x|<\epsilon}|x|^2 R_n(\rd x)=0, \mbox{ and }
\lim_{N\rightarrow\infty}\limsup_{n\rightarrow\infty} \int_{|x|>N}|x|^\gamma R_n(\rd x)=0.
\end{eqnarray*}
\end{lemma1}

\section{Inversions of L\'evy Measures}\label{sec: Duals}

\begin{defn}
Fix $\beta\in[0,2]$. If $M\in \mathfrak M^\beta$ we call the measure $M^\beta$ the $\beta$-inversion of $M$    
if $M^\beta(\{0\})=0$ and
\begin{eqnarray*}
M^\beta(A) = \int_{\mathbb R^d}1_A\left(\frac{x}{|x|^2}\right)|x|^{2+\beta}M(\rd x), \ \ \ A\in\mathfrak B(\mathbb R^d).
\end{eqnarray*}
\end{defn}

In \cite{Sato:2007}, the dual of an infinitely divisible distribution $\mu=ID_0(M,b)$ was defined to be the distribution $ID_0(M^0,-b)$.  Later, in \cite{Sato:2012} and \cite{Sato:Ueda:2012}, this was renamed the inversion of $\mu$. Thus the $0$-inversion of $M$ is the L\'evy measure of the inversion of $\mu$.  

It is straightforward to see that for $M\in\mathfrak M^\beta$ we have $M^\beta\in \mathfrak M^\beta$,
\begin{eqnarray}\label{eq: dual of dual for ID}
(M^\beta)^\beta = M,
\end{eqnarray}
and
\begin{eqnarray}
\int_{|x|>1}|x|^2 M(\rd x)<\infty &\Longleftrightarrow& \int_{|x|\le1}|x|^\beta M^\beta(\rd x)<\infty.
\end{eqnarray}
We now relate the convergence of a sequence of measures in $\mathfrak M^\beta$ to the convergence of the sequence of their $\beta$-inversions.

\begin{prop}\label{prop: conv of M iff conv of M*}
Fix $\beta\in[0,2]$, and let $M_0,M_1,M_2,\dots\in \mathfrak M^\beta$.\\
1. $M_n\conv M_0$ on $\mathbb R^d_0$ if and only if $M^\beta_n\conv  M^\beta_0$ on $\mathbb R^d_0$.\\
2. We have
\begin{eqnarray*}
\lim_{\epsilon\downarrow0}\limsup_{n\rightarrow\infty}\int_{|x|<\epsilon}|x|^2 M_n(\rd x)=0
\end{eqnarray*}
if and only if
\begin{eqnarray*}
\lim_{N\rightarrow\infty}\limsup_{n\rightarrow\infty}\int_{|x|>N }|x|^\beta M^\beta_n(\rd x)=0.
\end{eqnarray*}
\end{prop}

\begin{proof}
The second part follows immediately from the definition of $\beta$-inversions. To show the first part let $f:\mathbb R^d\mapsto\mathbb R$ be a bounded continuous function vanishing on a neighborhood of zero and on a neighborhood of infinity. The function $g(x) = f\left(\frac{x}{|x|^2}\right)|x|^{2+\beta}$ is also a continuous and bounded mapping of $\mathbb R^d$ into $\mathbb R$, vanishing on a neighborhood of zero and on a neighborhood of infinity. Thus, if $M^\beta_n\conv  M_0^\beta$ on $\mathbb R^d_0$ then
\begin{eqnarray*}
\lim_{n\rightarrow\infty}\int_{\mathbb R^d}f(x)M_n(\rd x) &=& \lim_{n\rightarrow\infty}\int_{\mathbb R^d}g(x)M^\beta_n(\rd x)\\
&=& \int_{\mathbb R^d}g(x)M_0^\beta(\rd x) = \int_{\mathbb R^d}f(x)M_0(\rd x),
\end{eqnarray*}
and $M_n\conv  M_0$ on $\mathbb R^d_0$. The other direction follows by \eqref{eq: dual of dual for ID}. \qed
\end{proof}

Our next result relates the regular variation of a L\'evy measure to the regular variation of its $\beta$-inversion. Before giving our results, we define regularly varying L\'evy measures. More information on regularly varying measures can be found in e.g.\ \cite{Meerschaert:Scheffler:2001}, \cite{Hult:Lindskog:2006}, or \cite{Resnick:2007}. However, our  formulation is somewhat different from those. 

\begin{defn}\label{defn: reg var for measures at infty}
Fix $\rho\le0$, $a\in\{0,\infty\}$, and $M\in\mathfrak M^0$ such that $M\ne0$. If $a=\infty$ assume further that $M$ has an unbounded support. $M$ is said to be regularly varying at $a$ with index $\rho$ if there is a finite, non-zero Borel measure $\sigma$ on $\mathbb S^{d-1}$ such that for any $t>0$ and any $D\in\mathfrak B(\mathbb S^{d-1})$ with $\sigma(\partial D)=0$
$$
\lim_{r\rightarrow a}\frac{M\left(|x|>r t, \frac{x}{|x|}\in D\right)}{M(|x|>r)} = t^{\rho}\frac{\sigma(D)}{\sigma(\mathbb S^{d-1})}.
$$
When this holds we write $M\in RV^a_\rho(\sigma)$.
\end{defn}

Note that, in the above, $\sigma$ is not unique.  It is only defined up to multiplication by a positive constant.

\begin{prop}\label{prop: long and short reg var}
Fix $\beta\in[0,2]$, $M\in\mathfrak M^\beta$, and $\rho\in(-2-\beta,0)$. If $\sigma\ne0$ is a finite Borel measure on $\mathbb S^{d-1}$  then
\begin{eqnarray}\label{eq: equiv of reg var for duals}
M^\beta \in RV^\infty_{\rho}(\sigma) \Longleftrightarrow M\in RV^0_{-(\rho+2+\beta)}(\sigma).
\end{eqnarray}
Moreover if $\ell\in RV^\infty_0$ then
\begin{eqnarray}\label{eq: M rv assymp}
M(|x|>t, x/|x|\in D) \sim \sigma(D)t^{-\rho-2-\beta}\ell(1/t) \ \mbox{as} \ t\downarrow0
\end{eqnarray}
for all $D\in\mathfrak B(\mathbb S^{d-1})$ with $\sigma(\partial D)=0$ if and only if
\begin{eqnarray}\label{eq: M* rv assymp}
M^\beta(|x|>t, x/|x|\in D) \sim \frac{\rho+2+\beta}{|\rho|}\sigma(D)t^{\rho}\ell(t) \ \mathrm{as} \ t\rightarrow\infty
\end{eqnarray}
for all $D\in\mathfrak B(\mathbb S^{d-1})$ with $\sigma(\partial D)=0$.
\end{prop}

\begin{proof}
Note that the equivalence between \eqref{eq: M rv assymp} and \eqref{eq: M* rv assymp} implies \eqref{eq: equiv of reg var for duals}, thus we just need to show that \eqref{eq: M rv assymp} holds if and only if \eqref{eq: M* rv assymp} holds. For every $D\in\mathfrak B(\mathbb S^{d-1})$ with $\sigma(\partial D)=0$, let
$$
V_D(t) = M\left(|x|>1/t,\frac{x}{|x|}\in D\right)\ \mbox{and}\ V^\beta_D(t) = M^\beta\left(|x|>t,\frac{x}{|x|}\in D\right).
$$
Assume that \eqref{eq: M rv assymp} holds for all $D\in\mathfrak B(\mathbb S^{d-1})$ with $\sigma(\partial D)=0$. This means that
$$
V_D(t)\sim \sigma(D)t^{\rho+2+\beta}\ell(t) \ \ \mathrm{as} \ t\rightarrow\infty.
$$
If $\sigma(D)>0$ then $V_D\in RV_{\rho+2+\beta}^\infty$, and since
$$
V_D(t) = \int_{\left[|x|<t,\frac{x}{|x|}\in D\right]}|x|^{2+\beta} M^\beta(\rd x) \mbox{ and } V^\beta_D(t) = \int_{\left[|x|>t,\frac{x}{|x|}\in D\right]} M^\beta (\rd x),
$$
Theorem 5.3.11 in \cite{Meerschaert:Scheffler:2001}  (or Theorem 2 in Section VIII.9 of \cite{Feller:1971}) implies that
$$
V^\beta_D(t) \sim t^{-2-\beta} V_D(t)\frac{2+\beta+\rho}{|\rho|} \ \mbox{as} \ t\rightarrow\infty.
$$
If $\sigma(D)=0$, fix $\epsilon>0$, let $M_\epsilon(\rd x) = \epsilon M(\rd x)$, and $M_D(\rd x)=1_D\left(\frac{x}{|x|}\right)M(\rd x)$.  Thus \eqref{eq: M rv assymp} implies
$$
\int_{|x|<t}|x|^{2+\beta}\left(M^\beta_D+M_\epsilon^\beta\right)(\rd x) \sim \epsilon\sigma(\mathbb S^{d-1}) t^{\rho+2+\beta}\ell(t) \ \mbox{as} \ t\rightarrow\infty,
$$
and, as before, Theorem 5.3.11 in \cite{Meerschaert:Scheffler:2001} implies that
$$
V^\beta_D(t) + \epsilon V^\beta_{\mathbb S^{d-1}}(t) \sim \epsilon \sigma(\mathbb S^{d-1})\frac{2+\beta+\rho}{|\rho|}t^{\rho}\ell(t) \ \mbox{as}\ t\rightarrow\infty.
$$
Thus
$$
\lim_{t\rightarrow\infty} \frac{V^\beta_D(t)}{t^{\rho}\ell(t)} \le \lim_{\epsilon\downarrow0}\lim_{t\rightarrow\infty} \frac{V^\beta_D(t)+ \epsilon V^\beta_{\mathbb S^{d-1}}(t)}{t^{\rho}\ell(t)} = \lim_{\epsilon\downarrow0}\epsilon \sigma(\mathbb S^{d-1})\frac{2+\beta+\rho}{|\rho|}=0.
$$
Hence \eqref{eq: M* rv assymp} holds for all $D\in\mathfrak B(\mathbb S^{d-1})$ with $\sigma(\partial D)=0$. The proof of the other direction is similar.  \qed
\end{proof}

\section{Convergence of Sequences in $ID_0$}\label{sec: sequences}

In this section we extend Proposition \ref{prop: conv of M iff conv of M*} to weak convergence of sequences of distribution in $ID_0$ and $\ts$.

\begin{prop}\label{prop: conv iff conv of beta dual for ID}
1. Fix $\beta\in[0,2]$ and let $M_0,M_1,M_2,\dots\in\mathfrak M^\beta$.  If $X_n\sim ID_0(M_n,b_n)$ and $X_n'\sim ID_0(M_n^\beta,b_n)$ for $n=0,1,2,\dots$ then
\begin{align}
X_n\cond X_0 \mbox{ and}\ \lim_{N\rightarrow\infty}&\limsup_{n\rightarrow\infty}\int_{|x|>N}|x|^\beta M_n(\rd x)=0 \label{eq: converg a}
\end{align}
if and only if
\begin{align}\label{eq: converg b}
X_n'\cond X_0'  \mbox{ and}\ \lim_{N\rightarrow\infty}&\limsup_{n\rightarrow\infty}\int_{|x|>N}|x|^\beta M^\beta_n(\rd x)=0.
\end{align}
2.  Fix $\beta\in[\gamma,2]$ and let $R_0,R_1,R_2,\dots\in\mathfrak M^\beta$. If $Y_n\sim \ts(R_n,b_n)$ and $Y_n'\sim \ts(R_n^\beta,b_n)$ for $n=0,1,2,\dots$  then
\begin{align}
Y_n\cond Y_0 \mbox{ and}\ \lim_{N\rightarrow\infty}&\limsup_{n\rightarrow\infty}\int_{|x|>N}|x|^\beta R_n(\rd x)=0 \label{eq: converg a}
\end{align}
if and only if
\begin{align}\label{eq: converg b2}
Y_n'\cond Y_0'  \mbox{ and}\ \lim_{N\rightarrow\infty}&\limsup_{n\rightarrow\infty}\int_{|x|>N}|x|^\beta R^\beta_n(\rd x)=0. 
\end{align}
\end{prop}

\begin{proof}
This result is an immediate consequence of Proposition \ref{prop: conv of M iff conv of M*}, Lemma \ref{lemma: conv ID} (or Lemma \ref{lemma: conv TS} in the case of Part 2), and \eqref{eq: dual of dual for ID}.\qed
\end{proof}

Combining this with Lemma \ref{lemma: conv ID} and Lemma \ref{lemma: conv TS} gives the following, the first part of which was previously given in Proposition 2.1 of \cite{Sato:2012}.

\begin{cor}\label{corr: conv iff conv of dual for ID}
1. If for $n=0,1,2,\dots$ we have $X_n\sim ID_0(M_n,b_n)$ and $X_n'\sim ID_0(M_n^0,b_n)$ then
\begin{eqnarray*}
X_n\cond X_0 \ \Longleftrightarrow X'_n\cond X_0'.
\end{eqnarray*}
2. If  for $n=0,1,2,\dots$ we have $Y_n\sim \ts(R_n,b_n)$ and $Y_n'\sim \ts(R_n^\gamma,b_n)$ then
\begin{eqnarray*}
Y_n\cond Y_0 \ \Longleftrightarrow Y_n'\cond Y_0'.
\end{eqnarray*}
\end{cor}

Note that, by definition, $R\in\mathfrak M^\gamma$ and thus $R^\gamma$ is always defined. Assume that for $n=0,1,2,\dots$ we have $Y_n\sim\ts(R_n,b_n)$ such that $Y_n\cond Y_0$.  Since $\ts(R_n,b_n) = ID_0(M_n,b)$, where we get $M_n$ from $R_n$ by \eqref{eq:levy m}, Corollary \ref{corr: conv iff conv of dual for ID} implies that if for each $n=0,1,2,\dots$ we have  $X_n'\sim ID_0(M_n^0,b_n)$ and $Y_n'\sim \ts(R_n^\gamma,b_n)$ then $X_n'\cond X_0'$ and $Y_n'\cond Y_0'$.  However, the distributions of $X_n'$ and $Y_n'$ are, in general, very different.  In fact the distribution of $Y_n'$ is necessarily in $\ts$, while the distribution of $X_n'$ is, in general, not an element of this class.  This last fact follows from Theorem 4.6 in \cite{Sato:2012}.

We end this section by specializing Corollary \ref{corr: conv iff conv of dual for ID} to the case of convergence to an infinite variance stable distribution. First, recall that, for $\eta\in(0,2)$, an $\eta$-stable distribution is a distribution $\mu=ID_0(M,b)$ where 
\begin{eqnarray}\label{eq: levy measure of stable}
M(A) = \int_{\mathbb S^{d-1}}\int_0^\infty 1_A(ur)r^{-1-\eta}\rd r\sigma(\rd u), & A\in\mathfrak B(\mathbb R^d)
\end{eqnarray}
for some finite, non-zero Borel measure $\sigma$ on $\mathbb S^{d-1}$. We denote this distribution by $S_\eta(\sigma,b)$.  For details about stable distributions see \cite{Samorodnitsky:Taqqu:1994}. Note that for $\beta\in[0,\eta)$
\begin{eqnarray*}
M^\beta(A) 
= \int_{\mathbb S^{d-1}}\int_0^\infty1_A(ur)r^{-1-(2+\beta-\eta)}\rd r\sigma(\rd u), & A\in\mathfrak B(\mathbb R^d),
\end{eqnarray*}
Thus $ID_0(M^\beta,b)=S_{2+\beta-\eta}(\sigma,b)$. In \cite{Grabchak:2012a} it was shown that if $\eta\in(\gamma,2)$ then  $S_\eta(\sigma,b)=\ts(R,b)$ where $R(\rd x) = K_{\eta,\alpha,p}^{-1}M(\rd x)$ with $M(\rd x)$ as given by \eqref{eq: levy measure of stable} and $K_{\eta,\alpha,p}=\int_0^\infty t^{\eta-\alpha-1}e^{-t^p}\rd t$.  Thus for $\eta\in(\gamma,2)$ we have  $R^\gamma(\rd x) = K_{\eta,\alpha,p}^{-1}M^\gamma(\rd x)$. Note that, in this case, $\ts(R^\gamma,b)=S_{2+\gamma-\eta}(\sigma',b)$ where
\begin{eqnarray}\label{eq: sigma'}
\sigma'(\rd u)=\frac{K_{(2+\gamma-\eta),\alpha,p}}{K_{\eta,\alpha,p}}\sigma(\rd u).
\end{eqnarray}
These facts, combined with Corollary \ref{corr: conv iff conv of dual for ID} give the following result.

\begin{cor}\label{corr: conv to stable for dual for ID}
1. Fix $\eta\in(0,2)$. If  for $n=0,1,2,\dots$ we have $X_n\sim ID_0(M_n,b_n)$ and $X_n'\sim ID_0(M_n^0,b_n)$ then $X_n\cond S_\eta(\sigma,b)$ if and only if $X_n'\cond S_{2-\eta}(\sigma,b)$.\\
2. Fix $\eta\in(\gamma,2)$. If  for $n=0,1,2,\dots$ we have $Y_n\sim \ts(R_n,b_n)$ and $Y_n'\sim \ts(R_n^\gamma,b_n)$ then $Y_n\cond S_\eta(\sigma,b)$ if and only if $Y_n'\cond S_{2+\gamma-\eta}(\sigma',b)$, where $\sigma'$ is given by \eqref{eq: sigma'}.
\end{cor}

\section{Long and Short Time Behavior For L\'evy Processes}\label{sec: long and short}

In this section we use the tools that we have developed to derive an equivalence between long and short time behavior of certain L\'evy processes.

\begin{thrm}\label{thrm: long iff short behavior}
1. Fix $\eta\in(0,2)$ and let $\{X_t:t\ge0\}$ and $\{X_t':t\ge0\}$ be L\'evy processes with $X_1 \sim ID_0(M,c)$ and $X_1'\sim ID_0(M^0,d)$. There exist functions $a_t$ and $\zeta_t$ such that
\begin{eqnarray}\label{eq: long time dual for ID}
a_t\left(X_t-\zeta_t\right)\cond S_\eta(\sigma,0) \  \mbox{as}\ t\rightarrow\infty
\end{eqnarray}
if and only if there exist functions $b_t$ and $\xi_t$ with
\begin{eqnarray}\label{eq: short time dual for ID}
b_t \left(X_t' -\xi_t\right) \cond S_{2-\eta}(\sigma,0) \  \mbox{as}\ t\downarrow0.
\end{eqnarray}
Moreover, when this holds we have $b_t \sim \left[(1/t)h^{-1}(1/t)\right]^{1/2}$ as $t\downarrow0$, where $h(t)$ is any invertible function with $h(t)\sim t^{-1}a_t^{-2}$ as $t\rightarrow\infty$.\\
2. Fix $\eta\in(\gamma,2)$ and let $\{Y_t:t\ge0\}$ and $\{Y_t':t\ge0\}$ 
be L\'evy processes with $Y_1 \sim \ts(R,c)$ and $Y_1'\sim \ts(R^\gamma,d)$. There exist functions $a_t$ and $\zeta_t$ such that
\begin{eqnarray}\label{eq: long time dual}
a_t\left(Y_t-\zeta_t\right)\cond  S_\eta(\sigma,0) \  \mbox{as}\ t\rightarrow\infty
\end{eqnarray}
if and only if there exist functions $b_t$ and $\xi_t$ with
\begin{eqnarray}\label{eq: short time dual}
b_t \left(Y_t'-\xi_t\right) \cond S_{2+\gamma-\eta}(\sigma,0)  \  \mbox{as}\ t\downarrow0.
\end{eqnarray}
Moreover, when this holds we have $b_t \sim \kappa^{-1/\eta}\left[(1/t)h_\gamma^{-1}(1/t)\right]^{1/(2+\gamma)}$ as $t\downarrow0$ , 
where $h_\gamma(t)$ is any invertable function with $h_\gamma(t)\sim t^{-1}a_t^{-2-\gamma}$ as $t\rightarrow\infty$ and $\kappa = K_{(2+\gamma-\eta),\alpha,p}/K_{\eta,\alpha,p}$.
\end{thrm}

From the standard theory of summation of iid random variables (see e.g.\ \cite{Feller:1971} or  \cite{Meerschaert:Scheffler:2001}) it follows that $a\in RV^\infty_{-1/\eta}$, thus the functions $t^{-1}a_t^{-2}$ and $t^{-1}a_t^{-2-\gamma}$ are regularly varying at infinity with a positive index of regular variation. This implies that they are asymptotically equivalent to an invertible function and hence $h$ and $h_\gamma$ are well defined. It is straightforward to see that in the first part $b\in RV^0_{-1/(2-\eta)}$ and in the second part $b\in RV^0_{-1/(2+\gamma-\eta)}$. 

\begin{proof}
We only prove the first part as the proof of the second part is similar. By Slutsky's Theorem it suffices to show that the result holds when $a_t=\left[th(t)\right]^{-1/2}$ and $b_t =\left[(1/t)h^{-1}(1/t)\right]^{1/2}$. Note that $b_t = 1/a_{h^{-1}(1/t)}$. Define 
$$
M'_t(B) = t\int_{\mathbb R^d}1_B(x a_t)M(\rd x) \ \ \mbox{and} \ \ M''_t(B) = t\int_{\mathbb R^d}1_B(x b_t)M^0(\rd x),\ \ B\in\mathfrak B(\mathbb R^d).
$$
These are, respectively, the L\'evy measures of $a_t\left(X_t-\zeta_t\right)$ and $b_t \left(X'_t-\xi_t\right)$. 

Assume that \eqref{eq: long time dual for ID} holds. Without loss of generality, we assume that $\zeta_t$ is such that $a_t\left(X_t-\zeta_t\right)\sim ID_0(M'_t,0)$ for every $t>0$. Since
\begin{eqnarray*}
\left(M'_t\right)^0(B) &=& \int_{\mathbb R^d}1_B\left(\frac{x}{|x|^2}\right) |x|^2M'_t(\rd x)\\
&=& ta_t^2\int_{\mathbb R^d}1_B\left(\frac{x}{|x|^2} a_t^{-1}\right) |x|^2M(\rd x)\\
&=& t a_t^2\int_{\mathbb R^d}1_B\left(x a_t^{-1}\right)M^0(\rd x),
\end{eqnarray*}
Corollary \ref{corr: conv to stable for dual for ID}  implies that
\begin{eqnarray*}
\frac{1}{a_t} \left(X'_{t a_t^2}-q_t\right) \cond  S_{2-\eta}(\sigma,0) \ \ \mbox{as} \ t\rightarrow\infty,
\end{eqnarray*}
where $q_t$ is such that $\frac{1}{a_t} \left(X'_{t a_t^2}-q_t\right) \sim ID_0\left((M'_t)^0,0\right)$. From here the result follows since
\begin{eqnarray*}
\lim_{t\rightarrow\infty}  \frac{1}{a_t} \left(X'_{t a_t^2}-q_t \right)&=& \lim_{t\downarrow0} \frac{1}{a_{1/t}} \left(X'_{t^{-1} a_{1/t}^2}-q_{1/t} \right) \\
&=& \lim_{t\downarrow0}  \frac{1}{a_{1/t}} \left(X'_{1/h(1/t)}-q_{1/t} \right)\\
&=&  \lim_{u\downarrow0} \frac{1}{a_{h^{-1}(1/u)}}\left( X'_{u} -q_{h^{-1}(1/u)} \right)= \lim_{u\downarrow0}b_u \left( X'_u-\xi_u \right),
\end{eqnarray*}
where the third line follows by the substitution $u=1/h(1/t)$ and $\xi_u=q_{h^{-1}(1/u)}$.

Conversely, assume that \eqref{eq: short time dual for ID} holds. Without loss of generality, we assume that $\xi_t$ is such that $b_t\left(X_t'-\xi_t\right)\sim ID_0\left((M''_t)^0,0\right)$ for every $t>0$. As before, since 
\begin{eqnarray*}
\left(M''_t\right)^0(B) = t b_t^2\int_{\mathbb R^d}1_B\left(x b_t^{-1}\right)M(\rd x),
\end{eqnarray*}
Corollary \ref{corr: conv to stable for dual for ID} implies that
\begin{eqnarray*}
\frac{1}{b_t} \left(X_{t b_t^2}-q'_t\right) \cond  S_\eta(\sigma,0) \ \ \mbox{as} \ t\downarrow0,
\end{eqnarray*}
where $q'_t$ is such that $\frac{1}{b_t} \left(X_{t b_t^2}-q'_t\right) \sim ID_0\left(M''_t,0\right)$. 
The result follows from the fact that
\begin{eqnarray*}
\lim_{t\downarrow0} \frac{1}{b_t} \left( X_{t b_t^2}-q_t\right) &=& \lim_{t\rightarrow\infty}  \frac{1}{b_{1/t}} \left(X_{t^{-1} b_{1/t}^2} -q_{1/t}\right)\\
&=& \lim_{t\rightarrow\infty} a_{h^{-1}(t)} \left( X_{h^{-1}(t)}-q_{1/t}\right) = \lim_{u\rightarrow\infty} a_u \left( X_u-\zeta_u\right),
\end{eqnarray*}
where $\zeta_u=q_{1/h(u)}$ and the  third equality follows by the substitution $u=h^{-1}(t)$.\qed
\end{proof}

We now derive self-contained conditions for a L\'evy process to converge in distribution to an infinite variance stable distribution when time approaches zero. It does not appear that this has been previously addressed for the multivariate case. However, in the univariate case, conditions in a slightly different form, are given in \cite{Maller:Mason:2008}.

\begin{thrm}\label{thrm: short time for Levy}
Fix $\eta\in(0,2)$ and let $\sigma$ be a finite, nonzero Borel measure on $\mathbb S^{d-1}$. Let $\{X_t:t\ge0\}$ be a L\'evy Process with $X_1\sim ID_0(M,b)$.\\
1. There exist functions $a_t>0$ and $\zeta_t$ such that
\begin{eqnarray}\label{eq: sort time for nongaus}
a_t \left(X_t-\zeta_t\right) \cond S_\eta(\sigma,0) \mbox{ as } t\downarrow0
\end{eqnarray}
if and only if $M\in RV^0_{-\eta}(\sigma)$.\\
2. If $\eta\in(\gamma,2)$ and $X_1\sim \ts(R,b)$ then there exist functions $a_t>0$ and $\zeta_t$ such that \eqref{eq: sort time for nongaus} holds  
if and only if $R\in RV^0_{-\eta}(\sigma)$. 
\end{thrm}

\begin{proof}
First we prove part 1. Theorem \ref{thrm: long iff short behavior} implies that \eqref{eq: sort time for nongaus} holds if and only if the distribution $ID_0(M^0,0)$ is in the domain of attraction of $S_{2-\eta}(\sigma,0)$.  By standard results (see e.g.\ \cite{Rvaceva:1962} or \cite{Meerschaert:Scheffler:2001}) this holds if and only if the distribution $ID(M^0,0)$ is an element of $RV_{-(2-\eta)}^\infty(\sigma)$. This holds if and only if $M^0\in RV_{-(2-\eta)}^\infty(\sigma)$ (see \cite{Hult:Lindskog:2006}). From here Part 1 follows by Proposition \ref{prop: long and short reg var}. The proof of the second part is similar, this time we must find necessary and sufficient conditions for the distribution $\ts(R^\gamma,0)$ to be an element of $RV_{-(2+\gamma-\eta)}^\infty(\sigma)$.  This holds if and only if $R^\gamma\in RV_{-(2+\gamma-\eta)}^\infty(\sigma)$ (see \cite{Grabchak:2012a}). From here part 2 follows by Proposition \ref{prop: long and short reg var}. \qed
\end{proof}

Combining the two parts of this theorem gives the following short time analogue of Theorem 5 in \cite{Grabchak:2012a}.

\begin{cor}
Let $R$ be the Rosi\'nski measure of a $p$-tempered $\alpha$-stable distribution, and let $M$ be the L\'evy measure of this distribution (that is, we get $M$ from $R$ by \eqref{eq:levy m}). If $\sigma$ is a finite, nonzero Borel measure on $\mathbb S^{d-1}$ and $\eta\in(\gamma,2)$ then
$$
R\in RV^0_{-\eta}(\sigma) \Longleftrightarrow M\in RV^0_{-\eta}(\sigma).
$$
\end{cor}

\begin{remark}\label{remark: case alpha=0}
We now turn to the case $\alpha=0$ and $p>0$.  For a measure $R$ to be the Rosi\'nski measure of some $p$-tempered $0$-stable distribution, $R$ must satisfy the condition that 
$R\in\mathfrak M^0 \mbox{ and } \int_{|x|>1} \log|x|R(\rd x)<\infty$ (see \cite{Grabchak:2012a}). Let $\mathfrak M^\lgg$ be the class of measures that satisfy this condition. The natural definition of the inversion of $R\in\mathfrak M^\lgg$ appears to be to let $R^\lgg(\{0\})=0$ and
$$
R^\lgg(A) = \int_{\mathbb R^d}1_A\left(\frac{x}{|x|^2}\right)|x|^2(1+|\log|x||)^{\kappa(|x|)} R(\rd x), \ \ A\in\mathfrak B(\mathbb R^d),
$$
where $\kappa(t) = 1$ if $t\ge1$ and $\kappa(t)=-1$ if $t\in[0,1)$. It is not difficult to see that $R^\lgg\in\mathfrak M^\lgg$ and $(R^\lgg)^\lgg=R$. All of the results in Sections \ref{sec: Duals} and \ref{sec: sequences} have a version for this case. However, we are not able to show the corresponding results from Section \ref{sec: long and short}.
\end{remark}

\section*{Acknowledgments}
This work was supported, in part, by funds provided by the University of North Carolina at Charlotte.  Much of the research for this paper was done while the author was a PhD student working with Professor Gennady Samorodnitsky. Professor Samorodnitsky's comments and support are gratefully acknowledged. The author wishes to thank the two anonymous referees whose detailed comments let to a great improvement in the presentation of this paper.

\end{document}